\documentclass{amsart}

\newtheorem{Theorem}{Theorem}[section]
\newtheorem{Corollary}[Theorem]{Corollary}
\newtheorem{Lemma}[Theorem]{Lemma}

\theoremstyle{remark}

\begin{document}

\title[Exceptional surgeries on alternating knots]%
{All exceptional surgeries on alternating knots are integral surgeries}

\author{Kazuhiro Ichihara}

\address{School of Mathematics Education,
Nara University of Education,
Takabatake-cho, Nara, 630-8528, Japan}

\email{ichihara@nara-edu.ac.jp}

\urladdr{http://mailsrv.nara-edu.ac.jp/~ichihara/index.html}



\begin{abstract}
We show that all non-trivial exceptional surgeries on 
hyperbolic alternating knots in the 3-sphere are integral surgeries. 
\end{abstract}

\maketitle

\section{Introduction}

By a \textit{Dehn surgery} on a knot 
(i.e., an embedded circle in a 3-manifold), 
we mean the following operation 
to create a new $3$-manifold 
from a given one and a given knot; 
remove an open tubular neighborhood of the knot, 
and glue a solid torus back. 
Here and hereafter in this paper, 
all 3-manifolds are assumed to be orientable. 

The well-known Hyperbolic Dehn Surgery Theorem 
due to Thurston \cite[Theorem 5.8.2]{Th} says that, 
each hyperbolic knot 
(i.e., a knot with hyperbolic complement) 
admits only finitely many Dehn surgeries yielding non-hyperbolic manifolds. 
In view of this, such finitely many exceptions 
are called \textit{exceptional surgeries}. 

In this paper, we consider exceptional surgery on 
hyperbolic \textit{alternating knots} in the 3-sphere $S^3$. 
A knot in $S^3$ is called alternating if 
it admits a diagram with alternatively arranged 
over-crossings and under-crossings running along it. 
Note that Menasco showed in \cite{M} that 
an alternating knot in $S^3$ is hyperbolic unless it is a $(2,p)$-torus knot; 
that is, the knot isotoped to the $(2,p)$-curve 
on the standard embedded torus in $S^3$.

Here let us recall fundamental terminologies about Dehn surgery. 
See \cite{R} in details for example. 
As usual, by a \textit{slope}, we mean the isotopy class of 
a non-trivial unoriented simple closed curve on a torus. 
Consider the slope on the peripheral torus of a knot $K$ 
which is represented by the curve 
identified with the meridian of the attached solid torus via the surgery. 
Then we can see that Dehn surgery on $K$ is characterized by the slope, 
which we call the \textit{surgery slope}. 
When $K$ is a knot in $S^3$, 
by using the standard meridian-longitude system, 
slopes on the peripheral torus are parametrized by 
rational numbers with $1/0$. 
For example, the meridian of $K$ corresponds to $1/0$ and the longitude to $0$. 
By the \textit{trivial} Dehn surgery on $K$ in $S^3$, 
we mean the Dehn surgery on $K$ along the meridional slope $1/0$. 
Thus it yields $S^3$ again, which is obviously exceptional, when $K$ is hyperbolic. 
We say that a Dehn surgery on $K$ in $S^3$ is \textit{integral} 
if it is along an integral slope. 
This means that the curve representing the surgery slope 
runs longitudinally once. 

Then our main theorem is;

\begin{Theorem}\label{Thm}
On a hyperbolic alternating knot in the $3$-sphere, 
all non-trivial exceptional surgeries are integral. 
\end{Theorem}

In general, it is conjectured that, 
on a hyperbolic knot in $S^3$, 
all non-trivial exceptional surgeries are integral or half-integral. 
See \cite[Problem 1.77, (A) Conjectures: (3)]{K}. 
Only known example of hyperbolic knots in $S^3$ 
admitting non-integral exceptional surgeries are 
the knots given by Eudave-Mu\~{n}oz in \cite{E}. 
Also see Subsection \ref{32}. 
Thus Theorem \ref{Thm} implies that 
the knots given by Eudave-Mu\~{n}oz in \cite{E} are all non-alternating.

Also, as an immediate corollary, we have the following; 

\begin{Corollary}\label{Cor}
A hyperbolic alternating knot in the $3$-sphere 
admits at most $10$ exceptional surgeries. 
\end{Corollary}

\begin{proof}
Let $K$ be a hyperbolic alternating knot in the $3$-sphere. 
Then, by Theorem~\ref{Thm}, 
all non-trivial exceptional surgeries are integral. 
On the other hand, as an immediate corollary to \cite[Theorem 1.1]{I}, 
it is shown that; 
on any hyperbolic knot in $S^3$, 
there are at most $9$ integral exceptional surgeries. 
Thus, together with the trivial surgery, 
$K$ admits at most $10$ exceptional surgeries. 
\end{proof}

This concerns the famous Gordon conjecture; 
there exist at most 10 exceptional surgeries on each hyperbolic knot. 
See \cite[Problem 1.77]{K}. 
Previously, the sharpest known bound was 12, 
which is obtained as a corollary of the so-called ``6-theorem'' 
given by Agol \cite[Theorem 8.1]{A} and Lackenby \cite{L} independently. 
Recently, in \cite{A2}, Agol announced that 
there are at most finitely many hyperbolic knots 
which admit more than eight exceptional surgeries. 
And, very recently, in \cite{LM}, 
Lackenby and Meyerhoff announced the affirmative answer to this conjecture 
by a combination of new geometric techniques and a rigorous computer-assisted calculation.
It should be mentioned that 
Corollary \ref{Cor} is worth little after Lackenby and Meyerhoff \cite{LM}, 
but they do not show the integrality of non-trivial exceptional surgeries.

\section*{Acknowledgments}

The author was partially supported by 
Grant-in-Aid for Young Scientists (B), No. 18740038, 
Ministry of Education,Culture,Sports,Science and Technology, Japan.
He would like to thank Shigeru Mizushima and Tetsuya Abe 
for their comments on the arguments about Lemma~\ref{lem4}. 
He also thanks to the referee for his/her careful reading.

\section{Proof of Theorem~\ref{Thm}}

We start with recalling a classification of exceptional surgeries. 
As a consequence of 
the famous Geometrization Conjecture, 
raised by Thurston in \cite[section 6, question 1]{Th2}, 
and established by recent Perelman's works \cite{P1, P2, P3}, 
all closed orientable $3$-manifolds are classified as; 
reducible (i.e., containing 2-spheres not bounding 3-balls), 
toroidal (i.e., containing incompressible tori), 
Seifert fibered (i.e., foliated by circles), 
or hyperbolic (i.e., admitting a complete Riemannian metric with constant sectional curvature $-1$). 
See \cite{S} for a survey. 
Thus exceptional surgeries are also divided into three types; 
reducible (i.e., yielding a reducible manifold), 
toroidal (i.e., yielding a toroidal manifold), or 
Seifert fibered (i.e., yielding a Seifert fibered manifold).

We first show the following lemma. 

\begin{Lemma}\label{lem1}
If a hyperbolic alternating knot $K$ has 
a connected prime alternating diagram $D$ satisfying $t(D) > 4$, 
then all non-trivial exceptional surgeries are integral. 
\end{Lemma}

The proof of this lemma 
heavily depends upon the result obtained by Lackenby in \cite{L}. 
Thus we prepare some terminologies defined and used there. 
Actually the following are simplified versions of the original definitions. 
See \cite{L} for full details. 
Let $D$ be a connected alternating diagram of a knot in $S^3$, 
which we view as a $4$-regular graph embedded in $S^2$, 
equipped with ``under-over" crossing information.
Then $D$ is called \textit{prime} 
if each simple closed curve in $S^2$ intersecting $D$ transversely 
in two points divides $S^2$ into two discs, 
one of which contains no crossings of $D$. 
The \textit{twist number} of the diagram $D$, 
denoted by $t(D)$, is defined as 
the number of \textit{twists}, which are either; 
maximal connected collections of bigon regions in $D$ arranged in a row 
or isolated crossings adjacent to no bigon regions.

\begin{proof}[Proof of Lemma~\ref{lem1}]
Suppose that a hyperbolic alternating knot $K$ has 
a connected prime alternating diagram $D$. 
Then the exterior of $K$ can be given 
the \textit{canonical angled spine} arising from $D$. 
See \cite[Section 4]{L} for the definition of the angled spine, and 
see \cite[Section 5]{L} for the construction of 
the canonical angled spine arising from $D$. 
From such an angled spine structure, 
the \textit{combinatorial length} of a slope 
on (the peripheral torus of) $K$ can be defined. 
See \cite[Section 4]{L} for its definition. 
Then it follows from \cite[Theorem 4.9]{L} that 
the combinatorial length of the slope $p/q$ on $K$ is at least $|q| \cdot t(D) \cdot \pi/4$. 

On the other hand, also in \cite{L}, 
Lackenby established 
a combinatorial analogue of the so-called ``$2\pi$-theorem" as \cite[Theorem 5.4]{L}. 
The following is a simplified version of his theorem, 
together with the affirmative answer to the Geometrization conjecture: 
Let $K$ be a hyperbolic knot in a closed orientable $3$-manifold, 
and $r$ a slope with combinatorial length more than $2\pi$ 
with respect to some angled spine in the exterior of $K$. 
Then Dehn surgery on $K$ along $r$ 
must yield a closed hyperbolic $3$-manifold. 

Therefore, on a hyperbolic alternating knot $K$ 
with a connected prime alternating diagram $D$ satisfying $t(D) > 4$, 
Dehn surgery along a non-integral slope (i.e., slope $p/q$ with $|q| \ge 2$) 
must yield a closed hyperbolic $3$-manifold, 
that is, such a surgery cannot be exceptional. 
\end{proof}

In the case where $t(D) \le 4$, we have the following: 

\begin{Lemma}\label{lem2}
If a hyperbolic alternating knot $K$ in $S^3$ has 
a connected prime alternating diagram $D$ satisfying $t(D) \le 4$, 
then $K$ must be an arborescent knot. 
\end{Lemma}

We can actually determine 
all the possible prime alternating diagrams. 

Here we recall definitions of an arborescent knot and its type. 
See \cite{W1} for full details. 
By a \textit{tangle}, we mean 
a pair with a 3-ball and properly embedded arcs. 
From two arcs of rational slope 
drawn on the boundary of a pillowcase-shaped 3-ball, 
one can obtain a tangle, which is called a \textit{rational tangle}. 
A tangle obtained by 
putting rational tangles together in a horizontal way 
is called a \textit{Montesinos tangle}. 
An \textit{arborescent tangle} is then defined as 
a tangle that can be obtained 
by summing several Montesinos tangles 
together in an arbitrary order. 

Suppose that 
a knot $K$ in $S^3$ is obtained by closing a tangle $T$. 
If $T$ is a Montesinos tangle, then we call $K$ a \textit{Montesinos knot}, and 
if $T$ is an arborescent tangle, then we call $K$ an \textit{arborescent knot}. 
For a Montesinos knot, 
the number of rational tangles 
forming the corresponding Montesinos tangle 
is called the \textit{length} of the Montesinos knot. 
We denote by $M(r_1,r_2,\cdots,r_n)$ 
a Montesinos knot constructed from rational tangles 
corresponding to rational numbers $r_1,r_2,\cdots,r_n$. 
In particular, $M(1/q_1, 1/q_2, \cdots, 1/q_n)$ with 
integers $q_1, q_2, \cdots, q_n$ is called a \textit{pretzel knot} of $n$-strands. 

In \cite{W1}, Wu divides all arborescent knots into three types: 
By the \textit{type I} knots, 
we mean two-bridge knots or Montesinos knots of length 3. 
A \textit{bridge index} of a knot in $S^3$ is defined as 
the minimal number of local maxima (or local minima) up to ambient isotopy. 
Thus a knot with bridge index $2$ is called a \textit{two-bridge knot}. 
As remarked before, Menasco showed in \cite{M} that 
a two-bridge knot is hyperbolic unless it is a $(2,p)$-torus knot. 
A knot of \textit{type II} is defined as the union of two Montesinos tangles, 
each of which is formed by two rational tangles 
corresponding to $1/2$ and a non-integer. 
All the other arborescent knots are called of \textit{type III}.

\begin{proof}[Proof of Lemma~\ref{lem2}]
Suppose that a hyperbolic alternating knot $K$ in $S^3$ has 
a connected prime alternating diagram $D$ satisfying $t(D) \le 4$. 
By regarding twists in $D$ (i.e., 
maximal connected collections of bigon regions in $D$ 
or isolated crossings adjacent to no bigon regions) as 
fat vertices, we have a planar embedding of 
a $4$-regular graph with at most four vertices. 
Note that when $t(D)>1$, 
such a graph has no loops 
otherwise the graph $D$ would be non-prime. 

Then we can tabulate all planar embeddings of such graphs 
as illustrated in Figure~\ref{fig1} (the case of $t(D) \le 3$) 
and Figure~\ref{fig2} (the case of $t(D) = 4$). 
This is done by using elementary diagrammatic arguments, and so, we omit the details. 
In the figures, we denote the vertices by small white boxes and the edges by thick lines.

\begin{figure}[htb]
\unitlength 0.1in
\begin{picture}( 10.0000,  6.0000)(  2.0000,-14.0000)
%
\special{pn 8}%
\special{pa 600 1000}%
\special{pa 800 1000}%
\special{pa 800 1200}%
\special{pa 600 1200}%
\special{pa 600 1000}%
\special{fp}%
%
\special{pn 20}%
\special{pa 800 1000}%
\special{pa 1000 800}%
\special{pa 1200 800}%
\special{pa 1200 1400}%
\special{pa 1000 1400}%
\special{pa 800 1200}%
\special{pa 800 1200}%
\special{fp}%
%
\special{pn 20}%
\special{pa 600 1200}%
\special{pa 400 1400}%
\special{pa 200 1400}%
\special{pa 200 800}%
\special{pa 400 800}%
\special{pa 600 1000}%
\special{pa 600 1000}%
\special{fp}%
\end{picture}%
\hspace{1cm}
\unitlength 0.1in
\begin{picture}( 10.0000, 10.0000)(  6.0000,-14.0000)
%
\special{pn 8}%
\special{pa 1400 800}%
\special{pa 1600 800}%
\special{pa 1600 1000}%
\special{pa 1400 1000}%
\special{pa 1400 800}%
\special{fp}%
%
\special{pn 8}%
\special{pa 600 800}%
\special{pa 800 800}%
\special{pa 800 1000}%
\special{pa 600 1000}%
\special{pa 600 800}%
\special{fp}%
%
\special{pn 20}%
\special{pa 800 800}%
\special{pa 1000 600}%
\special{pa 1200 600}%
\special{pa 1400 800}%
\special{pa 1400 800}%
\special{fp}%
%
\special{pn 20}%
\special{pa 1400 1000}%
\special{pa 1200 1200}%
\special{pa 1000 1200}%
\special{pa 800 1000}%
\special{pa 800 1000}%
\special{fp}%
%
\special{pn 20}%
\special{pa 600 800}%
\special{pa 600 600}%
\special{pa 800 400}%
\special{pa 1400 400}%
\special{pa 1600 600}%
\special{pa 1600 800}%
\special{pa 1600 800}%
\special{fp}%
%
\special{pn 20}%
\special{pa 1600 1000}%
\special{pa 1600 1200}%
\special{pa 1400 1400}%
\special{pa 800 1400}%
\special{pa 600 1200}%
\special{pa 600 1000}%
\special{pa 600 1000}%
\special{fp}%
\end{picture}%
\hspace{1cm}
\unitlength 0.1in
\begin{picture}( 18.0000, 10.0000)(  6.0000,-12.0000)
%
\special{pn 8}%
\special{pa 800 800}%
\special{pa 600 800}%
\special{pa 600 600}%
\special{pa 800 600}%
\special{pa 800 800}%
\special{fp}%
%
\special{pn 8}%
\special{pa 1400 600}%
\special{pa 1600 600}%
\special{pa 1600 800}%
\special{pa 1400 800}%
\special{pa 1400 600}%
\special{fp}%
%
\special{pn 8}%
\special{pa 2200 600}%
\special{pa 2400 600}%
\special{pa 2400 800}%
\special{pa 2200 800}%
\special{pa 2200 600}%
\special{fp}%
%
\special{pn 20}%
\special{pa 2200 600}%
\special{pa 2000 400}%
\special{pa 1800 400}%
\special{pa 1600 600}%
\special{pa 1600 600}%
\special{fp}%
%
\special{pn 20}%
\special{pa 1400 600}%
\special{pa 1200 400}%
\special{pa 1000 400}%
\special{pa 800 600}%
\special{pa 800 600}%
\special{fp}%
%
\special{pn 20}%
\special{pa 600 600}%
\special{pa 600 400}%
\special{pa 800 200}%
\special{pa 2200 200}%
\special{pa 2400 400}%
\special{pa 2400 600}%
\special{pa 2400 600}%
\special{fp}%
%
\special{pn 20}%
\special{pa 2200 800}%
\special{pa 2000 1000}%
\special{pa 1800 1000}%
\special{pa 1600 800}%
\special{pa 1600 800}%
\special{fp}%
%
\special{pn 20}%
\special{pa 1400 800}%
\special{pa 1200 1000}%
\special{pa 1000 1000}%
\special{pa 800 800}%
\special{pa 800 800}%
\special{fp}%
%
\special{pn 20}%
\special{pa 600 800}%
\special{pa 600 1000}%
\special{pa 800 1200}%
\special{pa 2200 1200}%
\special{pa 2400 1000}%
\special{pa 2400 800}%
\special{pa 2400 800}%
\special{fp}%
\end{picture}%
\caption{}\label{fig1}
\end{figure}

\begin{figure}[htb]
\unitlength 0.1in
\begin{picture}( 20.0000, 16.3000)(  6.0000,-18.3000)
%
\special{pn 8}%
\special{pa 600 800}%
\special{pa 800 800}%
\special{pa 800 1000}%
\special{pa 600 1000}%
\special{pa 600 800}%
\special{fp}%
%
\special{pn 8}%
\special{pa 1400 800}%
\special{pa 1600 800}%
\special{pa 1600 1000}%
\special{pa 1400 1000}%
\special{pa 1400 800}%
\special{fp}%
%
\special{pn 8}%
\special{pa 2200 400}%
\special{pa 2400 400}%
\special{pa 2400 600}%
\special{pa 2200 600}%
\special{pa 2200 400}%
\special{fp}%
%
\special{pn 8}%
\special{pa 2200 1200}%
\special{pa 2400 1200}%
\special{pa 2400 1400}%
\special{pa 2200 1400}%
\special{pa 2200 1200}%
\special{fp}%
%
\special{pn 20}%
\special{pa 2200 600}%
\special{pa 2000 800}%
\special{pa 2000 1000}%
\special{pa 2200 1200}%
\special{pa 2200 1200}%
\special{fp}%
%
\special{pn 20}%
\special{pa 2400 600}%
\special{pa 2600 800}%
\special{pa 2600 1000}%
\special{pa 2400 1200}%
\special{pa 2400 1200}%
\special{fp}%
%
\special{pn 20}%
\special{pa 1400 800}%
\special{pa 1200 600}%
\special{pa 1000 600}%
\special{pa 800 800}%
\special{pa 800 800}%
\special{fp}%
%
\special{pn 20}%
\special{pa 800 1000}%
\special{pa 1000 1200}%
\special{pa 1200 1200}%
\special{pa 1400 1000}%
\special{pa 1400 1000}%
\special{fp}%
%
\special{pn 20}%
\special{pa 1600 800}%
\special{pa 1800 400}%
\special{pa 2200 400}%
\special{pa 2200 400}%
\special{fp}%
%
\special{pn 20}%
\special{pa 1600 1000}%
\special{pa 1800 1400}%
\special{pa 2200 1400}%
\special{pa 2200 1400}%
\special{fp}%
%
\special{pn 20}%
\special{pa 600 800}%
\special{pa 600 600}%
\special{pa 1000 200}%
\special{pa 2200 200}%
\special{pa 2400 400}%
\special{pa 2400 400}%
\special{fp}%
%
\special{pn 20}%
\special{pa 600 1000}%
\special{pa 600 1200}%
\special{pa 1000 1600}%
\special{pa 2200 1600}%
\special{pa 2400 1400}%
\special{pa 2400 1400}%
\special{fp}%
\put(14.0000,-20.0000){\makebox(0,0)[lb]{(a)}}%
\end{picture}%
\hspace{0.5cm}
\unitlength 0.1in
\begin{picture}( 26.0000, 12.3000)(  6.0000,-14.3000)
%
\special{pn 8}%
\special{pa 800 800}%
\special{pa 600 800}%
\special{pa 600 600}%
\special{pa 800 600}%
\special{pa 800 800}%
\special{fp}%
%
\special{pn 8}%
\special{pa 1400 600}%
\special{pa 1600 600}%
\special{pa 1600 800}%
\special{pa 1400 800}%
\special{pa 1400 600}%
\special{fp}%
%
\special{pn 8}%
\special{pa 2200 600}%
\special{pa 2400 600}%
\special{pa 2400 800}%
\special{pa 2200 800}%
\special{pa 2200 600}%
\special{fp}%
%
\special{pn 20}%
\special{pa 2200 600}%
\special{pa 2000 400}%
\special{pa 1800 400}%
\special{pa 1600 600}%
\special{pa 1600 600}%
\special{fp}%
%
\special{pn 20}%
\special{pa 1400 600}%
\special{pa 1200 400}%
\special{pa 1000 400}%
\special{pa 800 600}%
\special{pa 800 600}%
\special{fp}%
%
\special{pn 20}%
\special{pa 600 600}%
\special{pa 600 400}%
\special{pa 800 200}%
\special{pa 3000 200}%
\special{pa 3200 400}%
\special{pa 3200 600}%
\special{pa 3200 600}%
\special{fp}%
%
\special{pn 20}%
\special{pa 2200 800}%
\special{pa 2000 1000}%
\special{pa 1800 1000}%
\special{pa 1600 800}%
\special{pa 1600 800}%
\special{fp}%
%
\special{pn 20}%
\special{pa 1400 800}%
\special{pa 1200 1000}%
\special{pa 1000 1000}%
\special{pa 800 800}%
\special{pa 800 800}%
\special{fp}%
%
\special{pn 8}%
\special{pa 3000 600}%
\special{pa 3200 600}%
\special{pa 3200 800}%
\special{pa 3000 800}%
\special{pa 3000 600}%
\special{fp}%
%
\special{pn 20}%
\special{pa 3000 600}%
\special{pa 2800 400}%
\special{pa 2600 400}%
\special{pa 2400 600}%
\special{pa 2400 600}%
\special{fp}%
%
\special{pn 20}%
\special{pa 3000 800}%
\special{pa 2800 1000}%
\special{pa 2600 1000}%
\special{pa 2400 800}%
\special{pa 2400 800}%
\special{fp}%
%
\special{pn 20}%
\special{pa 600 800}%
\special{pa 600 1000}%
\special{pa 800 1200}%
\special{pa 3000 1200}%
\special{pa 3200 1000}%
\special{pa 3200 800}%
\special{pa 3200 800}%
\special{fp}%
\put(18.0000,-16.0000){\makebox(0,0)[lb]{(b)}}%
\end{picture}%
\vspace{0.5cm}
\unitlength 0.1in
\begin{picture}( 20.0000, 14.0000)(  2.0000,-18.0000)
%
\special{pn 8}%
\special{pa 600 800}%
\special{pa 800 800}%
\special{pa 800 600}%
\special{pa 600 600}%
\special{pa 600 800}%
\special{fp}%
%
\special{pn 8}%
\special{pa 600 1400}%
\special{pa 800 1400}%
\special{pa 800 1600}%
\special{pa 600 1600}%
\special{pa 600 1400}%
\special{fp}%
%
\special{pn 8}%
\special{pa 1400 600}%
\special{pa 1600 600}%
\special{pa 1600 800}%
\special{pa 1400 800}%
\special{pa 1400 600}%
\special{fp}%
%
\special{pn 8}%
\special{pa 1400 1400}%
\special{pa 1600 1400}%
\special{pa 1600 1600}%
\special{pa 1400 1600}%
\special{pa 1400 1400}%
\special{fp}%
%
\special{pn 20}%
\special{pa 800 600}%
\special{pa 1000 400}%
\special{pa 1200 400}%
\special{pa 1400 600}%
\special{pa 1400 600}%
\special{fp}%
%
\special{pn 20}%
\special{pa 1400 800}%
\special{pa 1200 1000}%
\special{pa 1200 1200}%
\special{pa 1400 1400}%
\special{pa 1400 1400}%
\special{fp}%
%
\special{pn 20}%
\special{pa 800 1400}%
\special{pa 1000 1200}%
\special{pa 1000 1000}%
\special{pa 800 800}%
\special{pa 800 800}%
\special{fp}%
%
\special{pn 20}%
\special{pa 600 800}%
\special{pa 400 1000}%
\special{pa 400 1200}%
\special{pa 600 1400}%
\special{pa 600 1400}%
\special{fp}%
%
\special{pn 20}%
\special{pa 1600 1400}%
\special{pa 1800 1200}%
\special{pa 1800 1000}%
\special{pa 1600 800}%
\special{pa 1600 800}%
\special{fp}%
%
\special{pn 20}%
\special{pa 1600 600}%
\special{pa 1800 600}%
\special{pa 2000 800}%
\special{pa 2000 1400}%
\special{pa 1800 1600}%
\special{pa 1600 1600}%
\special{pa 1600 1600}%
\special{fp}%
%
\special{pn 20}%
\special{pa 1400 1600}%
\special{pa 1200 1800}%
\special{pa 1000 1800}%
\special{pa 800 1600}%
\special{pa 800 1600}%
\special{fp}%
%
\special{pn 20}%
\special{pa 600 1600}%
\special{pa 400 1600}%
\special{pa 200 1400}%
\special{pa 200 800}%
\special{pa 400 600}%
\special{pa 600 600}%
\special{pa 600 600}%
\special{fp}%
\put(22.0000,-12.0000){\makebox(0,0)[lb]{(c)}}%
\end{picture}%
\caption{}\label{fig2}
\end{figure}

By substituting suitable 
vertical or horizontal sequences of crossings into the white boxes, 
we can reconstruct the diagram $D$. 
Thus it suffices to show that 
all the prime connected diagrams obtained 
from the graphs in Figures~\ref{fig1} and \ref{fig2} 
represent arborescent knots. 

We first consider the graphs in Figure~\ref{fig1}. 
Those three graphs are corresponding to 
the cases $t(D) = 1$, $2$, $3$, respectively. 
In the case where $t(D) = 1$ (left in Figure~\ref{fig1}), 
it is clear that the resulting diagrams 
represent the unknot or $(2,p)$-torus knots, 
which are two-bridge knots. 
In the case where $t(D) = 2$ (center in Figure~\ref{fig1}), 
by substituting both vertical or both horizontal sequences of crossings 
into the two white boxes, 
we have the diagrams representing $(2,p)$-torus knots. 
On the other hand, 
if we substitute one vertical and one horizontal sequences of crossings 
into the two white boxes, 
we have the diagrams representing two-bridge knots. 
In the case where $t(D) = 3$ (right in Figure~\ref{fig1}), 
if we substitute three vertical sequences of crossings 
into all the three white boxes, 
then it represents a $3$-strand pretzel knot. 
Otherwise, we see that the diagrams so obtained 
all represent two-bridge knots (possibly $(2,p)$-torus knots). 

We next consider the graphs in Figure~\ref{fig2}. 
Those three graphs are corresponding to the case $t(D) = 4$. 
Note that, from the graph (c), 
we only have non-prime diagrams by substituting sequences of crossings. 
Thus we can ignore it. 

In the graph (a), there are four boxes; left two and right two. 
If we substitute both horizontal sequences of crossings into the left two boxes, 
or both vertical sequences of crossings into the right two boxes, 
then the diagrams so obtained have the twist number at most three, 
and so we can ignore these cases. 
If we substitute vertical and horizontal sequences of crossings 
into the left two and the right two white boxes simultaneously, 
we have the diagrams representing the unions of two rational tangles. 
Such diagrams actually represent two-bridge knots. 
Now suppose that 
we substitute two vertical sequences of crossings 
into the left two white boxes. 
If we further substitute vertical and horizontal sequences of crossings into the right two boxes, 
then we obtain a diagram of a Montesinos knot of length three. 
After taking a mirror image if necessary, 
it is denoted by $M(1/q_1, 1/q_2,1/(q_3 + 1/ q_4) )$ with $q_i > 1$ for $1 \le i \le 4$. 
Finally, if we further substitute two horizontal sequences of crossings into the right two white boxes, 
then we obtain a diagram of an arborescent knot of type II or III. 
In fact, if each one of the two sequences of crossing in the left and right white boxes is a full twist 
(i.e., a pair of crossings), then the knot so obtained is of type II. 
Otherwise all the knots so obtained are of type III. 

In the graph (b), there are four boxes arranged in line. 
If we substitute two horizontal sequences of crossings into the adjacent two boxes, 
then the diagrams so obtained have the twist number at most three, 
and so we can ignore these cases. 
If we substitute vertical and horizontal sequences of crossings alternatively 
into the four white boxes, 
then we have the diagrams representing the unions of two rational tangles. 
Such diagrams actually represent two-bridge knots. 
If we substitute four vertical sequences of crossings into the boxes, 
then we obtain a diagram representing a $4$-strand pretzel knot. 
Finally, it we substitute one 
horizontal sequences of crossings and 
three vertical sequences of crossings into the four boxes, 
then we have a diagram of a Montesinos knot of length three. 
Actually, after taking a mirror image, 
we have a Montesinos knot $M(1/p,1/q,n+1/r)$ with positive integers $p,q,r,n$. 

Consequently we have seen that 
all the prime connected diagrams obtained 
from the graphs in Figures~\ref{fig1} and \ref{fig2} 
represent arborescent knots. 
This completes the proof of Lemma~\ref{lem2}. 
\end{proof}

Thus we can divide the remaining arguments into three cases as follows.

\begin{Lemma}\label{lem3}
On a hyperbolic two-bridge knot in $S^3$, 
all non-trivial exceptional surgeries are integral. 
\end{Lemma}

\begin{proof}
In fact, in \cite{BW}, 
Brittenham and Wu gave a complete classification of 
the exceptional surgeries on hyperbolic two-bridge knots. 
From this, we can verify that; 
on a hyperbolic two-bridge knot, 
all non-trivial exceptional surgeries are integral. 
In fact, they showed that among two-bridge knots, 
only the twist knots (i.e., the knots with the corresponding partial fraction decomposition of length two) 
can admit non-trivial exceptional surgeries, which are actually all integral. 
\end{proof}

\begin{Lemma}\label{lem4}
On a hyperbolic alternating Montesinos knot of length $3$ in $S^3$, 
all non-trivial exceptional surgeries are integral. 
\end{Lemma}

The key ingredient in the proof of this lemma 
is using \textit{essential laminations} in 3-manifolds, 
defined by Gabai and Oertel in \cite{GO} as follows: 
We say a lamination $\lambda$ 
(i.e., a co-dimension one foliation of a closed subset of the ambient manifold) 
is an \textit{essential lamination} in a 3-manifold $M$ 
if it satisfies the following conditions: \\
(i) The inclusion of leaves of $\lambda$ into $M$ induces 
an injection between their fundamental groups.\\
(ii) The complement of $\lambda$ is irreducible.\\
(iii) The lamination $\lambda$ has no sphere leaves.\\
(iv) The lamination $\lambda$ is end-incompressible. \\

More about essential laminations, see \cite{G} for example.

\begin{proof}[Proof of Lemma~\ref{lem4}]
In an unpublished preprint \cite{D2}, 
Delman gave a construction of essential lamination 
in a Montesinos knot exterior. 
See \cite{DR} for a part of his construction. 
Actually if the Montesinos knot we consider has the form $M(1/p,r_1,r_2)$ 
with $p$ and all the denominators of $r_i$'s are all odd, 
then the construction given in \cite{DR} can be applied in our case. 
Also see \cite{D1} for prototype of the construction in \cite{D2}. 
Further detailed explanations of 
his construction for alternating Montesinos knots 
will be appeared in \cite{IJM}. 

In particular, Delman showed in \cite{D2} that all the Montesinos knots, 
not of the form $M (x, 1/p, 1/q)$ where 
$x \in  \{ -1/2n, -1 \pm  1/2n, - 2 + 1/2n \}$ and $p, q$, and $n$ are positive integers, 
admit essential laminations in their exteriors, which survive after all non-trivial Dehn surgeries. 
By the claim given in \cite[Section 4, 2nd paragraph]{T}, 
we see that these exceptions are all non-alternating knots. 
Precisely he claimed that 
a Montesinos knot is alternating if and only if 
its reduced Montesinos diagram, introduced in \cite{LT}, is alternating. 
And actually the reduced Montesinos diagrams of 
the Montesinos knots above are all non-alternating. 
See \cite{T} and \cite{LT} for detail. 

Thus, all the alternating Montesinos knots have 
such essential laminations in the exteriors. 
By examining his construction, we can verify that 
each essential lamination so constructed admits 
at least two disjoint, nonparallel annuli 
properly embedded in the complement of the lamination 
having the following property: 
One boundary component is the meridian of the knot 
and the other lies in some leaf of the lamination. 
Having been shown in \cite[Section 2]{B}, 
the existence of such a pair of annuli guarantees that 
non-trivial non-integral Dehn surgery on the knot cannot be exceptional. 
%
%
Also see \cite[Section 3]{W3} about this arguments. 
\end{proof}

\begin{Lemma}\label{lem5}
On a hyperbolic arborescent knot of type $II$ or $III$ in $S^3$, 
all non-trivial exceptional surgeries are integral. 
\end{Lemma}

\begin{proof}
In \cite[Theorem 4.4]{W1}, Wu showed that; 
all non-trivial non-integral surgeries on arborescent knots of type II 
give hyperbolic manifolds. 
Equivalently, an arborescent knot of type II has only integral non-trivial exceptional surgeries. 
Also, in \cite[Theorem 3.6]{W1}, Wu showed that; 
all non-trivial surgeries on arborescent knots of type III give hyperbolic manifolds. 
Equivalently, an arborescent knot of type III has no non-trivial exceptional surgeries. 
\end{proof}

\begin{proof}[Proof of Theorem~\ref{Thm}]
Let $K$ be a hyperbolic alternating knot in the $3$-sphere. 
By Lemma~\ref{lem1}, 
if $K$ has a connected prime alternating diagram $D$ satisfying $t(D) > 4$, 
then all non-trivial non-integral surgeries on $K$ give hyperbolic manifolds. 
Otherwise, by Lemma~\ref{lem2}, $K$ must be an arborescent knot. 
Following Wu \cite{W1}, we divide arborescent knots into of type I, II, or III, 
and furthermore, 
divide type I knots into two-bridge knots or Montesinos knots of length three. 
Then we obtain that; 
all non-trivial exceptional surgeries are integral 
on a hyperbolic two-bridge knot by Lemma~\ref{lem3}, 
on a Montesinos knot of length three by Lemma~\ref{lem4}, 
and on an arborescent knot of type II or III by Lemma~\ref{lem5}, respectively. 
This completes the proof. 
\end{proof}

\section{Remarks}

In this section, we collect some known facts about reducible/toroidal surgeries on alternating knots. 
As far as the author knows, 
there are no explicit studies on Seifert surgeries on alternating knots.

\subsection{}\label{AltTor}

We begin with considering reducible surgeries on alternating knots. 

In fact, 
it is shown by Menasco and Thistlethwaite in \cite[Corollary 1.1]{MT} that 
no Dehn surgeries on a hyperbolic alternating knot in $S^3$ 
yield reducible manifolds. 
We here include an outline of their arguments for completeness. 
They studied essential surfaces 
(i.e., incompressible and boundary-incompressible) 
properly embedded in alternating knot exteriors, and established 
that the following holds for an essential surface $F$; 
\begin{equation}\label{MT}
- \chi (F) \ge \dfrac{1}{8} b \beta (n+2) , 
\end{equation}
where $b$ denotes the denominator of the boundary slope of $F$, 
$\beta$ the number of boundary components of $F$, and 
$n$ the twist-crossing number of the standard diagram of the knot. 
Thus if $F$ is of genus $0$, then 
$ \beta - 2 \ge \dfrac{1}{8} b \beta (n+2) $. 
From $\beta \ge 1$ and $b \ge 1$, this implies that $n \le 5$. 
However, in \cite{MT}, they determined 
alternating knots with twist-crossing number up to five 
(\cite[Figure 1.2]{MT}), and they are in fact all two-bridge knots. 
For two-bridge knots, it is shown by Hatcher and Thurston 
in \cite[Theorem 2(a)]{HT} that 
all non-trivial two-bridge knots other than $(2,p)$-torus knots 
have no reducible surgery. 

In general, F. Gonz\`{a}lez-Acu\~{n}a and H. Short conjectured in \cite{GS} that 
the only way to get a reducible 3-manifold by surgery on a knot in $S^3$ 
is to surger on a cable knot along the slope determined by the cabling annulus. 
This is now called the Cabling Conjecture, and still remaining open. 
See \cite[Problem 1.79]{K}.

\subsection{}\label{32}

We next consider toroidal surgeries on alternating knots. 
Such surgeries are actually completely classified as follows. 

In \cite{P}, based on the result obtained in \cite{MT}, 
Patton claimed that if an alternating knot admits a toroidal surgery, 
then it is either a 2-bridge knot or a $3$-strand pretzel knot. 
This is also achieved as \cite[Lemmas 3.1 and 3.3]{BoyerZhang}. 
We here include an outline of Patton's argument for completeness. 
Note that 
if a hyperbolic knot in $S^3$ yields a toroidal surgery, 
then the surgery slope actually becomes a boundary slope of 
an essential punctured torus. 
Then, from Equation~\eqref{MT}, we see that 
the denominator $b$ of the boundary slope of the punctured torus is at most 2. 
However, if $b=2$, then the twist-crossing number must be at most 2. 
Then we see that such knots are only $(2,p)$-torus knots, 
which are non-hyperbolic. 
If $b=1$, from Equation~\eqref{MT}, we see that 
the twist-crossing number of 
a hyperbolic alternating knot admitting toroidal surgery is at most 6. 
In \cite[Section 1]{P}, Patton asserts that 
a hyperbolic alternating knot with twist-crossing number at most 6 
must be a two-bridge knot or a Montesinos knot of length $3$. 
In the former case, in \cite[Section 2]{P}, 
by using the machinery developed in \cite{HT}, 
he showed that 
the 2-bridge knots have Conway forms of length two, 
i.e., each of the knots is either of genus one, 
or bounds a once punctured Klein bottle.
Also see \cite{BW}. 
In the latter case, he studied in detail in \cite[Section 3]{P}. 
He used the machinery developed in \cite{HO}, and showed that 
the Montesinos knots must be $3$-strand pretzel knots. 
Again, each of such knots is either of genus one, 
or bounds a once punctured Klein bottle.

On general Montesinos knots, 
in \cite{W4} together with his previous results, 
Wu gave a complete classification of toroidal surgeries. 
Furthermore, he also gave a complete classification 
of toroidal surgeries on large arborescent knots, in \cite{W5}.

Also remark that Gordon and Luecke proved in \cite{GL} that 
if a hyperbolic knot in $S^3$ admits a non-integral toroidal surgery, 
then the knot is one explicitly given by Eudave-Mu\~{n}oz in \cite{E}.

\bibliographystyle{amsplain}

\end{document}